\documentclass[11pt]{article}
\usepackage{amsfonts, amssymb, amsmath, amsthm, xcolor, graphicx}
\begin{document}
\newcommand{\m}{\textnormal{mult}}
\newcommand{\ra}{\rightarrow}
\newcommand{\la}{\leftarrow}
\renewcommand{\baselinestretch}{1.1}

\theoremstyle{plain}
\newtheorem{thm}{Theorem}[section]
\newtheorem{cor}[thm]{Corollary}
\newtheorem{con}[thm]{Conjecture}
\newtheorem{cla}[thm]{Claim}
\newtheorem{lm}[thm]{Lemma}
\newtheorem{prop}[thm]{Proposition}
\newtheorem{example}[thm]{Example}

\theoremstyle{definition}
\newtheorem{dfn}[thm]{Definition}
\newtheorem{alg}[thm]{Algorithm}
\newtheorem{prob}[thm]{Problem}
\newtheorem{rem}[thm]{Remark}

\renewcommand{\baselinestretch}{1.1}

\title{\bf Extensions of Barrier Sets to Nonzero Roots of the Matching Polynomials}
\author{
Cheng Yeaw Ku
\thanks{ Department of Mathematics, National University of Singapore, Singapore 117543. E-mail: matkcy@nus.edu.sg} \and K.B. Wong \thanks{
Institute of Mathematical Sciences, University of Malaya, 50603 Kuala Lumpur, Malaysia. E-mail:
kbwong@um.edu.my.} } \maketitle

\begin{abstract}\noindent
In matching theory, barrier sets (also known as Tutte sets) have been studied extensively due to its connection to maximum matchings in a graph. In this paper, we first define $\theta$-barrier sets. Our definition of a $\theta$-barrier set is slightly different from that of a barrier set. However we show that $\theta$-barrier sets and barrier sets have similar properties. In particular, we prove a generalized Berge's Formula and give a characterization for the set of all $\theta$-special vertices in a graph.
\end{abstract}

\bigskip\noindent
{\sc keywords:} matching polynomial, Gallai-Edmonds Decomposition, barrier sets, extreme sets

\section{Introduction}

All the graphs in this paper are simple and finite.

\begin {dfn}\label {I:D1}  An $r$-\emph {matching} in a graph $G$ is a set of $r$ edges, no two of which have a vertex in common. The number of $r$-matchings in $  G$ will be denoted by $p( G,r)$. We set $p(G,0)=1$ and define the \emph {matching polynomial} of $G$ by
\begin {equation}
\mu ( G,x)=\sum_{r=0}^{\lfloor n/2\rfloor} (-1)^rp(G,r)x^{n-2r}.\notag
\end {equation}
We shall denote the multiplicity of $\theta$ as a root of $\mu(G,x)$ by $\textnormal {mult} (\theta,G)$. Let $u\in V(G)$, the graph obtained from $G$ by deleting the vertex $u$ and all edges that contain $u$ will be denoted by $G\setminus u$. Inductively if $u_1,\dots, u_k\in V(G)$, $G\setminus u_1\dots u_k=(G\setminus u_1\dots u_{k-1})\setminus u_k$. Note that the order of which vertex is being deleted first is not important, that is, if $i_1,\dots, i_k$ is a permutation  of $1,\dots, k$, we have  $G\setminus u_1\dots u_k= G\setminus u_{1_1}\dots u_{i_k}$. Furthermore if  $X=\{u_1,\dots, u_k\}$, $G\setminus X=G\setminus u_1\dots u_k$.
\end {dfn}

The followings  are properties of $\mu (G,x)$.

\begin {thm}\label {I:T2} \textnormal {(Theorem 1.1 on p. 2 of \cite {G0})}
\begin {itemize}
\item [(a)] $\mu (  G\cup  H,x)=\mu (  G,x)\mu (  H,x)$ where $  G$ and $  H$ are disjoint graphs,
\item [(b)] $\mu (  G,x)=\mu (  G\setminus e,x)-\mu (  G\setminus uv, x)$ if $e=\{u,v\}$ is an edge of $  G$,
\item [(c)] $\mu (  G,x)=x\mu (  G\setminus u,x)-\sum_{i\sim u} \mu (  G\setminus ui,x)$ where $i\sim u$ means $i$ is adjacent to $u$,
\item [(d)] $\displaystyle \frac {d}{dx} \mu (  G,x)=\sum_{i\in V(  G)} \mu (  G\setminus i,x)$ where $V(  G)$ is the vertex set of $  G$.
\end {itemize}
\end {thm}

It is well known that all roots of $\mu(G,x)$ are real. Throughout, let $\theta$ be a real number and $\textnormal {mult} (\theta, G)$ denote the multiplicity of $\theta$ as a root of $\mu(G,x)$. In particular, $\textnormal {mult} (\theta, G)=0$ if and only if $\theta$ is not a root of $\mu(G,x)$. By Theorem 5.3 on p. 29 and Theorem 1.1 on p. 96 of \cite {G0}, one can easily deduce the following lemma.

\begin {lm}\label {P:L2} Let $  G$ be a graph  and $u\in V(  G)$. Then
\begin {equation}
\textnormal {mult} (\theta,   G)-1\leq \textnormal {mult} (\theta,   G\setminus u)\leq \textnormal {mult} (\theta,  G)+1.\notag
\end {equation}
\end {lm}

As a consequence of Lemma \ref  {P:L2}, we can classify the vertices in a graph with respect to $\theta$ as follows:

\begin {dfn}\label {P:D3}\textnormal {(see \cite [Section 3]{G})} For any $u\in V(  G)$,
\begin {itemize}
\item [(a)] $u$ is $\theta$-\emph {essential} if $\textnormal {mult} (\theta,   G\setminus u)=\textnormal {mult} (\theta,   G)-1$,
\item [(b)] $u$ is $\theta$-\emph {neutral} if $\textnormal {mult} (\theta,   G\setminus u)=\textnormal {mult} (\theta,   G)$,
\item [(c)] $u$ is $\theta$-\emph {positive} if $\textnormal {mult} (\theta,   G\setminus u)=\textnormal {mult} (\theta,   G)+1$.
\end {itemize}
Furthermore if $u$ is not $\theta$-essential but it is adjacent to some $\theta$-essential vertex, we say $u$ is $\theta$-special.
\end {dfn}

It turns out that $\theta$-special vertices play an important role in the Gallai-Edmonds Decomposition of a graph (see \cite {KC}). One of our main result is a characterization of the set of these vertices in terms of $\theta$-barriers.

Note that if $\textnormal {mult} (\theta,   G)=0$ then for any $u\in V(G)$, $u$ is either $\theta$-neutral or $\theta$-positive and no vertices in $G$ can be $\theta$-special. By Corollary 4.3 of \cite {G}, a $\theta$-special vertex is $\theta$-positive. Therefore
\begin {equation}
V(G)=D_{\theta}(G)\cup A_{\theta}(G)\cup P_{\theta}(G)\cup N_{\theta}(G),\notag
\end {equation}
where
\begin {itemize}
\item [] $D_{\theta}(G)$ is the set of all $\theta$-essential vertices in $G$,
\item [] $A_{\theta}(G)$ is the set of all $\theta$-special vertices in $G$,
\item [] $N_{\theta}(G)$ is the set of all $\theta$-neutral vertices in $G$,
\item [] $P_{\theta}(G)=Q_{\theta}(G)\setminus A_{\theta}(G)$, where  $Q_{\theta}(G)$ is the set of all $\theta$-positive vertices in $G$,
\end {itemize}
is a partition of $V(G)$.
Note that there is no $0$-neutral vertices. So $N_0(G)=\varnothing$ and $V(G)=D_{0}(G)\cup A_{0}(G)\cup P_{0}(G)$.

\begin {dfn}\label {P:D4}\textnormal {(see \cite [Section 3]{G})} A graph $G$ is said to be $\theta$-critical if all vertices in $G$ are $\theta$-essential and $\textnormal {mult} (\theta, G)=1$.
\end {dfn}

The Gallai-Edmonds Structure Theorem describes a certain canonical decomposition of $V(G)$ with respect to the zero root of $\mu (G,x)$. In \cite {KC}, Chen and Ku proved the Gallai-Edmonds Structure Theorem for graph with any root $\theta$.

\begin {thm}\label {P:T5}\textnormal {(Theorem 1.5 of \cite {KC})} Let $G$ be a graph with $\theta$ a root of $\mu (G,x)$. If $u\in A_{\theta} (G)$ then
\begin {itemize}
\item [(i)] $D_{\theta}(G\setminus u)=D_{\theta}(G)$,
\item [(ii)] $P_{\theta}(G\setminus u)=P_{\theta}(G)$,
\item [(iii)] $N_{\theta}(G\setminus u)=N_{\theta}(G)$,
\item [(iv)] $A_{\theta}(G\setminus u)=A_{\theta}(G)\setminus \{u\}$.
\end {itemize}
\end {thm}

\begin {thm}\label {P:T6}\textnormal {(Theorem 1.7 of \cite {KC})} If $G$ is connected and every vertex of $G$ is $\theta$-essential then $\textnormal {mult} (\theta,   G)=1$.
\end {thm}

By Theorem \ref {P:T5} and Theorem \ref {P:T6}, it is not hard to deduce the following whose proof is omitted. For convenience, a connected component will be called a component.

\begin {cor}\label {P:C7}\
\begin {itemize}
\item [(i)]  $A_{\theta}(G\setminus A_{\theta}(G))=\varnothing$, $D_{\theta}(G\setminus A_{\theta}(G))=D_{\theta}(G)$, $P_{\theta}(G\setminus A_{\theta}(G))=P_{\theta}(G)$, and $N_{\theta}(G\setminus A_{\theta}(G))=N_{\theta}(G)$.
\item [(ii)] $G\setminus A_{\theta}(G)$ has exactly $\vert A_{\theta}(G)\vert+\textnormal {mult} (\theta, G)$ $\theta$-critical components.
\item [(iii)] If $H$ is a component of $G\setminus A_{\theta}(G)$ then either $H$ is $\theta$-critical or $\textnormal {mult} (\theta, H)=0$.
\item [(iv)] The subgraph induced by $D_{\theta}(G)$ consists of all the  $\theta$-critical components in $G\setminus A_{\theta}(G)$.
\end {itemize}
\end {cor}

 Let $G$ be a graph. The number of odd components in $G$ is denoted by $c_{odd} (G)$.  Recall the following famous Berge's Formula.
\begin {thm}\label {P:T8} $\textnormal {mult} (0,G)=\max_{X\subseteq V(G)} c_{odd} (G\setminus X)-\vert X\vert$.
\end {thm}

\begin {dfn}\label {P:D9} Motivated by the Berge's Formula, a \emph {barrier set} is defined to be a set $X\subseteq V(G)$ for which $\textnormal {mult} (0,G)=c_{odd} (G\setminus X)-\vert X\vert$. An \emph {extreme} set is defined to be the set for which $\textnormal {mult} (0,G\setminus X)=\textnormal {mult} (0,G)+\vert X\vert$.
\end {dfn}

Properties of extreme and barrier sets can be found in \cite [Section 3.3]{Lo}. In fact a barrier set is an extreme set. An extreme set is not necessary a barrier set, but it can be shown that an extreme set is contained in some barrier set. In general the union or intersection of two barrier sets is not a barrier set. However it can be shown that the intersection of two (inclusionwise) maximal barriers set is a barrier set. $A_0(G)$ is a barrier and extreme set. It can be shown that $A_0(G)$ is in fact the intersection of all the maximal barrier sets in $G$. Here we extend this fact to $A_{\theta}(G)$:

\begin {thm} Suppose $N_{\theta}(G)=\varnothing$. Then $A_{\theta}(G)$ is the intersection of all maximal $\theta$-barrier sets in $G$.
\end {thm}

\section{Properties of $\theta$-barrier sets}

The number of $\theta$-critical components in $G$ is denoted by $c_{\theta} (G)$. An immediate consequence of part (a) of Theorem \ref {I:T2} and Theorem \ref{P:T6} is the following inequality which is used frequently.

\begin{eqnarray}
\textnormal {mult} (\theta, G\setminus X)\geq c_{\theta} (G\setminus X)~~~\textnormal{for any}~~X \subseteq V(G). \label{ee}
\end{eqnarray}

\noindent We prove the following analogue of Berge's Formula.

\begin {thm}\label {B:T1}\textnormal {[Generalized Berge's Formula]}
\begin {equation}
\textnormal {mult} (\theta, G)=\max_{X\subseteq V(G)} c_{\theta} (G\setminus X)-\vert X\vert.\notag
\end {equation}
\end {thm}

\begin {proof} We claim that, $c_{\theta} (G\setminus X)\leq \vert X\vert+\textnormal {mult} (\theta, G)$ for all $X\subseteq V(G)$. Suppose the contrary. Then $c_{\theta} (G\setminus X)>\vert X\vert+\textnormal {mult} (\theta, G)$ for some $X\subseteq V(G)$. Recall that $\textnormal {mult} (\theta, G\setminus X)\geq c_{\theta} (G\setminus X)$. Together with Lemma \ref {P:L2}, we have  $\textnormal {mult} (\theta, G)\geq \textnormal {mult} (\theta, G\setminus X)-\vert X\vert>\textnormal {mult} (\theta, G)$, a contradiction. Hence  $c_{\theta} (G\setminus X)\leq \vert X\vert+\textnormal {mult} (\theta, G)$ for all $X\subseteq V(G)$.

Now it is sufficient to show that there is a set $X\subseteq V(G)$ for which $\textnormal {mult} (\theta, G)=c_{\theta} (G\setminus X)-\vert X\vert$. By (ii) of Corollary  \ref {P:C7} and taking $X=A_{\theta}(G)$, we are done.
\end {proof}

\begin {dfn}\label {BP:D2}  Motivated by the Generalized Berge's Formula, we define a $\theta$-\emph {barrier set} to be a set $X\subseteq V(G)$ for which $ \textnormal {mult} (\theta,G)=c_{\theta} (G\setminus X)-\vert X\vert$.

We define a $\theta$-\emph {extreme set} to be a set $X\subseteq V(G)$ for which  $\textnormal {mult} (\theta,G\setminus X)=\textnormal {mult} (\theta,G)+\vert X\vert$.
\end {dfn}

Note that the definitions of $0$-extreme set and extreme set coincide. But the definitions of $0$-barrier set and barrier set are different. Our next proposition shows that a $0$-barrier set is a barrier set.

\begin {prop}\label {BP:P3}  A  $0$-barrier set is a barrier set.
\end {prop}

\begin {proof} Let $X$ be a $0$-barrier set. Then $c_{0} (G\setminus X)=\textnormal {mult} (0,G)+\vert X\vert$. Note that $c_{0} (G\setminus X)\leq c_{odd} (G\setminus X)$. Using Theorem \ref {P:T8}, we conclude that $c_{odd} (G\setminus X)=\textnormal {mult} (0,G)+\vert X\vert$. Hence $X$ is a barrier set.
\end {proof}

The converse of Proposition \ref {BP:P3}  is not true. In Figure 1, $X=\{u,v\}$ is a barrier set in $G$ but it is not a $0$-barrier set.

\begin {center}
\scalebox{0.7}[0.7]{\includegraphics{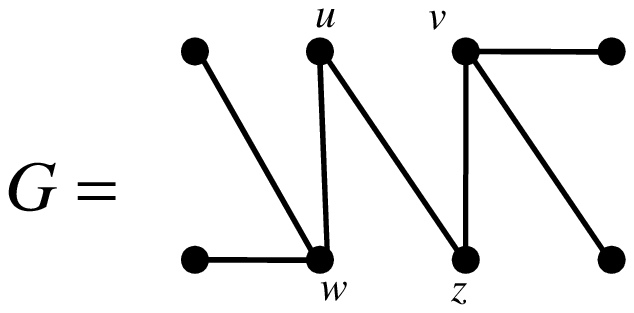}}

\textnormal {Figure 1.}
\end {center}
However we have a weak converse of Proposition \ref {BP:P3}.
\begin {prop}\label {BP:P3b}  A (inclusionwise) maximal barrier set is a maximal $0$-barrier set.
\end {prop}

\begin {proof} Let $X$ be a maximal barrier set. Note that  $\vert X\vert+\textnormal {mult} (0,G)\geq \textnormal {mult} (0,G\setminus X)\geq c_{odd} (G\setminus X)=\vert X\vert+\textnormal {mult} (0,G)$, where the first inequality follows from Lemma \ref {P:L2} and the last inequality follows from the fact that $X$ is a barrier set. Therefore, equality holds throughout whence $\textnormal {mult} (0,G\setminus X)=c_{odd} (G\setminus X)$ and 0 is a root of multiplicity $1$ in each of the odd components in $G\setminus X$.

We claim that an odd component in $G\setminus X$ is $0$-critical. Suppose the contrary. Let $H$ be an odd component in $G\setminus X$ and  $H$ is not $0$-critical. Then $A_0(H)\neq \varnothing$. Now $\textnormal {mult} (0,H)=1$. By (ii) of Corollary \ref {P:C7}, $c_{0} (H\setminus A_0(H))=\vert A_0(H)\vert+1$. Since $c_{0} (H\setminus A_0(H))\leq c_{odd} (H\setminus A_0(H))$, by Theorem \ref {P:T8}, we conclude that
$c_{odd} (H\setminus A_0(H))=\vert A_0(H)\vert+1$. Therefore $c_{odd} (G\setminus (X\cup A_0(H))=c_{odd} (G\setminus X)-1+c_{odd} (H\setminus A_0(H))=\vert X\vert+\textnormal {mult} (0,G)-1+\vert A_0(H)\vert+1=\vert X\cup A_0(H)\vert+\textnormal {mult} (0,G)$. But then $X\cup A_0(H)$ is a barrier set in $G$, a contrary to the maximality of $X$. Hence an odd component in $G\setminus X$ must be $0$-critical. This means that $c_{odd} (G\setminus X)=c_0 (G\setminus X)$ and $X$ is a $0$-barrier set. By Proposition \ref {BP:P3},  we conclude that $X$ must be a maximal $0$-barrier set.
\end {proof}

Now we shall study the properties of $\theta$-barrier and $\theta$-extreme sets.

\begin {lm}\label {BP:L3} A subset of a $\theta$-extreme set is a $\theta$-extreme set.
\end {lm}

\begin {proof} Let $X$ be an $\theta$-extreme set and $Y\subseteq X$. Now $\textnormal {mult} (\theta,G\setminus X)=\textnormal {mult} (\theta,G)+\vert X\vert$. By Lemma \ref {P:L2}, $\textnormal {mult} (\theta,G\setminus Y)\leq \textnormal {mult} (\theta,G)+\vert Y\vert$. If $Y$ is not $\theta$-extreme then  $\textnormal {mult} (\theta,G\setminus Y)<\textnormal {mult} (\theta,G)+\vert Y\vert$, and by Lemma \ref {P:L2} again, $\textnormal {mult} (\theta,G\setminus X)\leq \textnormal {mult} (\theta,G\setminus Y)+\vert X\setminus Y\vert< \textnormal {mult} (\theta,G)+\vert X\vert$, a contradiction. Hence a subset of an $\theta$-extreme set is $\theta$-extreme.
\end {proof}

\begin {lm}\label {BP:L4} If $X$ is a $\theta$-barrier [$\theta$-extreme] set and $Y\subseteq X$ then $X\setminus Y$ is a $\theta$-barrier [$\theta$-extreme] set in $G\setminus Y$.
\end {lm}

\begin {proof} Note that $c_{\theta} (G\setminus X)=\vert X\vert+\textnormal {mult} (\theta, G)$. By  Theorem \ref {B:T1} and Lemma \ref {P:L2}, $c_{\theta} (G\setminus X)\leq \vert X\setminus Y\vert+\textnormal {mult} (\theta, G\setminus Y)\leq \vert X\setminus Y\vert+\textnormal {mult} (\theta, G)+\vert Y\vert= \vert X\vert+\textnormal {mult} (\theta, G)$. Hence $c_{\theta} (G\setminus X)=\vert X\setminus Y\vert+\textnormal {mult} (\theta, G\setminus Y)$ and $X\setminus Y$ is a $\theta$-barrier set in $G\setminus Y$.
\end {proof}

\begin {lm}\label {BP:L6} Every $\theta$-extreme set of $G$ lies in a $\theta$-barrier set.
\end {lm}

\begin {proof} Let $X$ be a $\theta$-extreme set and $T=A_{\theta}(G\setminus X)\cup X$. Then
\begin {align}
c_{\theta} (G\setminus T) & =c_{\theta}(G\setminus (A_{\theta}(G\setminus X)\cup X)) \notag\\
& =c_{\theta}((G\setminus X)\setminus A_{\theta}(G\setminus X)) \notag\\
& =\vert A_{\theta}(G\setminus X)\vert +\textnormal {mult} (\theta, G\setminus X)\ \textnormal {(by (ii) of Corollary \ref {P:C7})}\notag\\
& =\vert A_{\theta}(G\setminus X)\vert +\textnormal {mult} (\theta, G)+\vert X\vert\ \textnormal {($X$ is $\theta$-extreme)}\notag\\
& =\vert T\vert +\textnormal {mult} (\theta, G),\notag
\end {align}
and hence $T$ is a $\theta$-barrier set.
\end {proof}

\begin {lm}\label {BP:L8a} Let $X$ be a $\theta$-barrier set. Then $X$ is a $\theta$-extreme set.
\end {lm}

\begin {proof} Recall from (\ref{ee}) that $\textnormal {mult} (\theta, G\setminus X)\geq c_{\theta} (G\setminus X)$. Since $c_{\theta} (G\setminus X)=\vert X\vert+\textnormal {mult} (\theta, G)$, by Lemma \ref {P:L2}, we have
\begin {equation}
\textnormal {mult} (\theta, G)\geq \textnormal {mult} (\theta, G\setminus X)-\vert X\vert\geq c_{\theta} (G\setminus X)-\vert X\vert=\textnormal {mult} (\theta, G).\notag
\end {equation}
Hence $\textnormal {mult} (\theta, G\setminus X)=\textnormal {mult} (\theta, G)+\vert X\vert$ and so $X$ is a $\theta$-extreme set.
\end {proof}

Note that in general a $\theta$-extreme set is not a $\theta$-barrier set. In Figure 1, $X_1=\{u\}$ is a $0$-extreme set but it is not a $0$-barrier set.

\begin {lm}\label {BP:L8b} Let $X$ be a $\theta$-barrier set and $H$ be a component of $G\setminus X$. Then either $H$ is $\theta$-critical or $\textnormal {mult} (\theta,H)=0$.
\end {lm}

\begin {proof}  Note that $c_{\theta} (G\setminus X)=\vert X\vert+\textnormal {mult} (\theta, G)$. By Lemma \ref {BP:L8a}, $X$ is a $\theta$-extreme set. Therefore  $\textnormal {mult} (\theta, G\setminus X)=\textnormal {mult} (\theta, G) +\vert X\vert=c_{\theta} (G\setminus X)$. Now if $H$ is not $\theta$-critical and $\textnormal {mult} (\theta,H)>0$, then by (\ref{ee}), $\textnormal {mult} (\theta, G\setminus X)> c_{\theta} (G\setminus X)$, a contradiction. Hence either $H$ is $\theta$-critical or $\textnormal {mult} (\theta,H)=0$.
\end {proof}

\begin {lm}\label {BP:L8} Let $X$ be a maximal  $\theta$-barrier set. Let $H$ be a component of $G\setminus X$ and $\textnormal {mult} (\theta,H)=0$. Then for all $u\in V(H)$, $u$ is $\theta$-neutral in $H$. Furthermore for all $Y\subseteq V(H)$ and $Y\neq \varnothing$, $c_{\theta}(H\setminus Y)\leq \vert Y\vert-1$.
\end {lm}

\begin {proof} Suppose $H$ has a $\theta$-positive vertex, say $u$. Then $\textnormal {mult} (\theta, H\setminus u)=1$. By (ii) of Corollary  \ref {P:C7}, $c_{\theta} ((H\setminus u)\setminus A_{\theta}(H\setminus u))=\vert A_{\theta}(H\setminus u)\vert+\textnormal {mult} (\theta, H\setminus u)=\vert A_{\theta}(H\setminus u)\vert+1$. But then
\begin {align}
c_{\theta} (G\setminus (X\cup \{u\}\cup A_{\theta}(H\setminus u)) &=c_{\theta} (G\setminus X)+c_{\theta} ((H\setminus u)\setminus A_{\theta}(H\setminus u))\notag\\
&=\vert X\vert+\textnormal {mult} (\theta, G)+\vert A_{\theta}(H\setminus u)\vert+1\notag\\
&=\vert X\cup \{u\}\cup A_{\theta}(H\setminus u)\vert+\textnormal {mult} (\theta, G),\notag
\end {align}
and so $X\cup \{u\}\cup A_{\theta}(H\setminus u)$ is a $\theta$-barrier in $G$, a contrary to the maximality of $X$. Hence
 for all $u\in V(H)$, $u$ is $\theta$-neutral in $H$.

Since $Y\neq \varnothing$, there is a $y\in Y$. Let $Y'=Y\setminus y$ and $H'=H\setminus y$. Note that $\textnormal {mult} (\theta,H\setminus y)=0$ since $y$ is  $\theta$-neutral in $H$. By Theorem \ref {B:T1}, $c_{\theta} (H'\setminus Y')\leq \vert Y'\vert$. Since $H\setminus Y=H'\setminus Y'$, we have $c_{\theta}(H\setminus Y)\leq \vert Y\vert-1$.
\end {proof}

\begin {lm}\label {BP:L9}  Let $G$ be $\theta$-critical. Then for all $Y\subseteq V(G)$ and $Y\neq \varnothing$, $c_{\theta}(G\setminus Y)\leq \vert Y\vert-1$.
\end {lm}

\begin {proof} Since $Y\neq \varnothing$, there is a $y\in Y$. Let $Y'=Y\setminus y$ and $G'=G\setminus y$. Note that  $\textnormal {mult} (\theta,G\setminus y)=0$ since $y$ is  $\theta$-essential in $G$. By Theorem \ref {B:T1}, $c_{\theta} (G'\setminus Y')\leq \vert Y'\vert$. Since $G\setminus Y=G'\setminus Y'$, we have $c_{\theta}(G\setminus Y)\leq \vert Y\vert-1$.
\end {proof}

In general the union or intersection of two $\theta$-barrier sets is not necessary a $\theta$-barrier set. In Figure 1, $X_2=\{u,v,w\}$ and $X_3=\{v,w,z\}$ are two $0$-barrier sets. But $X_2\cap X_3$ and $X_2\cup X_3$ are not a $0$-barrier set. However the intersection of two maximal $\theta$-barrier sets is a $\theta$-barrier set.

\begin {thm}\label {BP:T10} The intersection of two maximal $\theta$-barrier sets is a $\theta$-barrier set.
\end {thm}

\begin {proof} Let $X$ and $Y$ be two maximal $\theta$-barrier sets. Let $G_1,G_2,\dots, G_k$ be all the $\theta$-critical components of $G\setminus X$ and $H_1,H_2,\dots, H_m$ be all the components of $G\setminus Y$.  Note that $k=\vert X\vert+\textnormal {mult} (\theta,G)$. Let $X_i=X\cap V(H_i)$, $Y_i=Y\cap V(G_i)$ and $Z=X\cap Y$. By relabelling if necessary we may assume that $X_1,\dots, X_{m_1}\neq \varnothing$, $Y_1,\dots, Y_{k_1}\neq \varnothing$, but $X_{m_1+1}=\cdots=X_{m}=Y_{k_1+1}=\cdots=Y_k=\varnothing$, and also that $k_1\leq m_1$. Note that $G_{k_1+1},\dots, G_{k}$ are $\theta$-critical components in $(G\setminus X)\setminus Y$. So  each of them is contained in a component of $G\setminus Y$. Now let us count the number of $G_i$'s where $k_1+1\leq i\leq k$ that are contained in some $H_j$.

Suppose $m_1+1\leq j\leq m$. Then $H_j$ is a component in $(G\setminus X)\setminus Y$. So if $G_i\subseteq H_j$, we must have $G_i=H_j$. Furthermore $G_i$ is a component of $G\setminus Z$. By Theorem \ref {B:T1}, the number of such $G_i$'s is at most $c_{\theta} (G\setminus Z)\leq \vert Z\vert +\textnormal {mult} (\theta,G)$.

Suppose $1\leq j\leq m_1$. Let $G_{i_1},\dots, G_{i_t}$ be all the $G_i$'s that are contained in $H_j$. Then  $G_{i_1},\dots, G_{i_t}$ are $\theta$-critical components in $H_j\setminus X_j$. By Lemma \ref {BP:L8b}, $H_j$ is either  $\theta$-critical or $\textnormal {mult} (\theta,H)=0$. If $\textnormal {mult} (\theta,H)=0$, we have, by Lemma \ref {BP:L8}, $c_{\theta} (H_j\setminus X_j)\leq \vert X_j\vert-1$. If $H_i$ is $\theta$-critical, we have, by Lemma \ref {BP:L9}, $c_{\theta} (H_j\setminus X_j)\leq \vert X_j\vert-1$. Therefore in either cases, we have $t\leq \vert X_j\vert-1$.

The number of $G_i$'s where $k_1+1\leq i\leq k$ that are disjoint from $Y$ is at most
\begin {align}
c_{\theta} (G\setminus Z)+\sum_{j=1}^{m_1} (\vert X_j\vert-1) &\leq  \vert Z\vert +\textnormal {mult} (\theta,G)+\vert X\setminus Z\vert-m_1\notag\\
& =\vert X\vert +\textnormal {mult} (\theta,G)-m_1\notag\\
& =k-m_1\notag\\
& \leq k-k_1.\notag
\end {align}
Since this number is exactly $k-k_1$, we infer that equality must hold throughout. Hence $c_{\theta} (G\setminus Z)=\vert Z\vert +\textnormal {mult} (\theta,G)$ and $Z$ is a $\theta$-barrier set.
\end {proof}

\section{Characterizations of $A_{\theta}(G)$}

A characterization of $A_{\theta}(G)$  is that it is the minimal (inclusionwise) $\theta$-barrier set (see Theorem \ref {BP:T13b}). Furthermore if $N_{\theta}(G)=\varnothing$, we have another characterization of $A_{\theta}(G)$, that is, it is  the intersection of all maximal $\theta$-barrier sets in $G$ (see Theorem \ref {BP:T15}).

\begin {lm}\label {BP:L11} If $X$ is a $\theta$-barrier or a $\theta$-extreme set then $X\subseteq A_{\theta}(G)\cup P_{\theta}(G)$.
\end {lm}

\begin {proof} By Lemma \ref {BP:L8a}, we may assume $X$ is a $\theta$-extreme. Let $x\in X$. By Lemma \ref {BP:L3}, $\{x\}$ is a $\theta$-extreme set. Therefore $\textnormal {mult} (\theta, G\setminus x)=\textnormal {mult} (\theta, G)+1$ and $x$ is $\theta$-positive. So $x\in A_{\theta}(G)\cup P_{\theta}(G)$ and $X\subseteq A_{\theta}(G)\cup P_{\theta}(G)$.
\end {proof}

\begin {lm}\label {BP:L13} Let $X$ be a $\theta$-barrier set. If $X\subseteq A_{\theta}(G)$ then $X=A_{\theta}(G)$.
\end {lm}

\begin {proof} Note that $c_{\theta} (G\setminus X)=\textnormal {mult} (\theta, G)+\vert X\vert$. By Lemma \ref  {BP:L8b},  we conclude that $A_{\theta}(G\setminus X)=\varnothing$. By Theorem \ref {P:T5}, $A_{\theta}(G\setminus X)=A_{\theta}(G)\setminus X$. Hence $X=A_{\theta}(G)$.
\end {proof}

We shall require the following result of Godsil \cite {G}.

\begin {thm}\label {G:4.2} \textnormal {(Theorem 4.2 of \cite {G})}
Let $\theta$ be a root of $\mu(G,x)$ with non-zero multiplicity $k$
and let $u$ be a $\theta$-positive vertex in $G$. Then
\begin {itemize}
\item [(a)] if $v$ is $\theta$-essential in $G$ then it
is $\theta$-essential in $G \setminus u$;
\item [(b)] if $v$ is $\theta$-positive in $G$ then it
is $\theta$-essential or $\theta$-positive in $G \setminus u$;
\item [(c)] if $u$ is $\theta$-neutral in $G$ then it is
$\theta$-essential or $\theta$-neutral in $G \setminus u$.
\end {itemize}
\end {thm}

\begin {lm}\label {BP:L13a} Let $u\in P_{\theta}(G)$. Then $A_{\theta}(G)\subseteq A_{\theta}(G\setminus u)$.
\end {lm}

\begin {proof} If $A_{\theta}(G)=\varnothing$, then we are done. Suppose $A_{\theta}(G)\neq \varnothing$. Let $v\in  A_{\theta}(G)$. Then $v$ is adjacent to a $\theta$-essential vertex $w$. By Theorem \ref{G:4.2}, $w$ is $\theta$-essential in $G\setminus u$ and $v$ is either $\theta$-positive or $\theta$-essential in $G\setminus u$. Suppose $v$ is $\theta$-essential in $G\setminus u$. Then $\textnormal {mult} (\theta, G\setminus uv)=\textnormal {mult} (\theta, G)$. By Theorem \ref {P:T5}, $u\in P_{\theta}(G)=P_{\theta}(G\setminus v)$. Since $v$ is $\theta$-special in $G$, $v$ is $\theta$-positive in $G$ (see Corollary 4.3 of \cite {G}). So $\textnormal {mult} (\theta, G\setminus uv)=\textnormal {mult} (\theta, G)+2$, a contradiction. Therefore $v$ is $\theta$-positive  in $G\setminus u$. Since $v$ is adjacent to $w$, we must have $v\in A_{\theta}(G\setminus u)$. Hence   $A_{\theta}(G)\subseteq A_{\theta}(G\setminus u)$.
\end {proof}

\begin {thm}\label {BP:T13b} Let $X$ be a $\theta$-barrier set in $G$. Then $A_{\theta}(G)\subseteq X$. In particular, $A_{\theta}(G)$ is the minimal $\theta$-barrier set.
\end {thm}

\begin {proof} By Lemma \ref {BP:L11}, $X\subseteq A_{\theta}(G)\cup P_{\theta}(G)$. We shall prove the result by induction on $\vert X\cap P_{\theta}(G)\vert $. Suppose $\vert X\cap P_{\theta}(G)\vert=0$. Then $X\subseteq A_{\theta}(G)$ and by Lemma \ref {BP:L13}, $X=A_{\theta}(G)$. Suppose $\vert X\cap P_{\theta}(G)\vert\geq 1$. We may assume that if $X'$ is a $\theta$-barrier set in $G'$ with $\vert X'\cap P_{\theta}(G')\vert<\vert X\cap P_{\theta}(G)\vert$, then $A_{\theta}(G')\subseteq X'$.

Let $x\in X\cap P_{\theta}(G)$. By Lemma \ref {BP:L4}, $X'=X\setminus x$ is a $\theta$-barrier set in $G'=G\setminus x$. By Lemma \ref {BP:L11} and Lemma \ref {BP:L13a}, we have $X'\subseteq A_{\theta}(G')\cup P_{\theta}(G')$ and  $A_{\theta}(G)\subseteq A_{\theta}(G')$. Therefore $\vert X'\cap P_{\theta}(G')\vert<\vert X\cap P_{\theta}(G)\vert$. By induction $A_{\theta}(G')\subseteq X'$. Hence  $A_{\theta}(G)\subseteq X$.
\end {proof}

In general, $A_{\theta}(G)$ is not the intersection of all maximal $\theta$-barrier sets in $G$. For instance, in Figure 2, $\textnormal {mult} (\sqrt {3},G)=0$ and $A_{\sqrt {3}}(G)=\varnothing$. Now $\{u\}$ is the only maximal $\sqrt {3}$-barrier set. But $A_{\sqrt {3}}(G)\neq \{u\}$. However we can show that $A_{\theta}(G)$ is  the intersection of all maximal $\theta$-barrier sets in $G$ if $N_{\theta}(G)=\varnothing$.

\begin {center}
\scalebox{1}[1]{\includegraphics{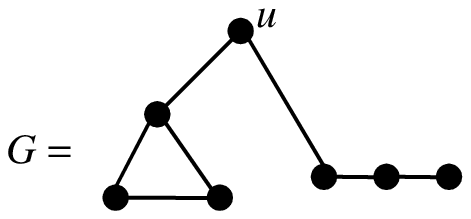}}

\textnormal {Figure 2.}
\end {center}

\begin {thm}\label {BP:T15} Suppose $N_{\theta}(G)=\varnothing$. Then $A_{\theta}(G)$ is the intersection of all maximal $\theta$-barrier sets in $G$.
\end {thm}

\begin {proof} By Theorem \ref {BP:T13b}, $A_{\theta}(G)$ is contained in the intersection of all maximal $\theta$-barriers in $G$. It is sufficient to show that for each $x\in V(G)\setminus A_{\theta}(G)$ there is a maximal barrier that does not contain $x$. If $x\in D_{\theta}(G)$, by Lemma \ref {BP:L11}, $x$ is not contained in any $\theta$-barriers and thus any maximal $\theta$-barriers. Suppose $x\in P_{\theta}(G)$. Then $x$ is contained in a component $H$ in $G\setminus A_{\theta}(G)$ with $\textnormal {mult} (\theta, H)= 0$. Note that $\vert V(H)\vert\geq 2$, for $x\in P_{\theta}(G)=P(G\setminus A_{\theta}(G))$ and $\textnormal {mult} (\theta, H\setminus x)=1$ (see Theorem \ref {P:T5}). By (c) of Theorem \ref {I:T2} and the fact that $\textnormal {mult} (\theta, H)= 0$, we deduce that there is a vertex $y\in V(H\setminus x)$ for which $\textnormal {mult} (\theta, H\setminus xy)=0$. Now $y\in P_{\theta}(G)$ for $N_{\theta}(G)=\varnothing$. Furthermore $x$ is $\theta$-essential in $H\setminus y$. Therefore $x\notin A_{\theta}(H\setminus y)$ and by (ii) of Corollary \ref {P:C7}, $c_{\theta} ((H\setminus y)\setminus A_{\theta}(H\setminus y))=\vert A_{\theta}(H\setminus y)\vert +1$. Hence
\begin {align}
c_{\theta} (G\setminus (A_{\theta}(G)\cup \{y\}\cup A_{\theta}(H\setminus y))) & =c_{\theta} (G\setminus A_{\theta}(G))+c_{\theta} ((H\setminus y)\setminus A_{\theta}(H\setminus y))\notag\\
& =\vert A_{\theta}(G)\vert+\textnormal {mult} (\theta, G)+\vert A_{\theta}(H\setminus y)\vert +1\notag\\
&=\vert A_{\theta}(G)\cup \{y\}\cup A_{\theta}(H\setminus y)\vert+\textnormal {mult} (\theta, G),\notag
\end {align}
and so $A_{\theta}(G)\cup \{y\}\cup A_{\theta}(H\setminus y)$ is a $\theta$-barrier set not containing $x$. Let $Z$ be a maximal $\theta$-barrier set containing $Y=A_{\theta}(G)\cup \{y\}\cup A_{\theta}(H\setminus y)$. By Lemma \ref {BP:L4}, $Z\setminus Y$ is a $\theta$-barrier set in $G\setminus Y$. Using Theorem \ref {P:T5} and the fact that $x$ is $\theta$-essential in $H\setminus y$, we can deduce that $x\in D_{\theta} (G\setminus Y)$. By Lemma \ref {BP:L11}, we conclude that $x\notin Z\setminus Y$ and hence $x\notin Z$. The proof of the theorem is completed.
\end {proof}

Since $N_0(G)=\varnothing$, by Theorem \ref {BP:T15} and Proposition \ref {BP:P3b}, we deduce the following classical result.

\begin {cor}\label {BP:C16}  \textnormal {(Theorem 3.3.15 of \cite {Lo})}
$A_{0}(G)$ is the intersection of all maximal barrier sets in $G$.
\end {cor}

\end{document}